\documentclass[12pt,reqno]{amsart}
\usepackage{amscd,amsmath,amsthm,amssymb}
\usepackage{color}
\usepackage{stmaryrd}
\usepackage{tikz-cd}
\newcommand{\arrowIn}{
\tikz \draw[-stealth] (-1pt,0) -- (1pt,0);
}
\newcommand{\arrowOut}{
\tikz \draw[-stealth] (1pt,0) -- (-1pt,0);
}
\usepackage{url}
\tikzset{empty/.style={black, fill=white}}

\pgfdeclarelayer{background}
\pgfdeclarelayer{foreground}
\pgfsetlayers{background,main,foreground}

\usepackage{latexsym}
\usepackage{amsfonts,amsmath,mathtools}
\usepackage{graphics}
\usepackage{float}
\usepackage{enumitem}

\def\NZQ{\mathbb}               

\def\ZZ{{\NZQ Z}}

\def\G{{\mathcal G}}

\def\pd{\textup{pd}}

\def\opn#1#2{\def#1{\operatorname{#2}}} 
\opn\chara{char} \opn\length{\ell} \opn\pd{pd} \opn\rk{rk}
\opn\projdim{proj\,dim} \opn\injdim{inj\,dim} \opn\rank{rank}
\opn\depth{depth} \opn\grade{grade} \opn\height{height}
\opn\embdim{emb\,dim} \opn\codim{codim}

\opn\Tr{Tr} \opn\bigrank{big\,rank}
\opn\superheight{superheight}\opn\lcm{lcm}
\opn\trdeg{tr\,deg}
\opn\reg{reg} \opn\lreg{lreg} \opn\ini{in} \opn\lpd{lpd}
\opn\size{size} \opn\sdepth{sdepth}
\opn\link{link}\opn\fdepth{fdepth}\opn\lex{lex}
\opn\tr{tr}
\opn\type{type}
\opn\gap{gap}
\opn\diam{diam}
\opn\Mod{Mod}

\opn\div{div} \opn\Div{Div} \opn\cl{cl} \opn\Cl{Cl}
\opn\Spec{Spec} \opn\Supp{Supp} \opn\supp{supp} \opn\Sing{Sing}
\opn\Ass{Ass} \opn\Min{Min}\opn\Mon{Mon}
\opn\Ann{Ann} \opn\Rad{Rad} \opn\Soc{Soc}
\opn\Im{Im} \opn\Ker{Ker} \opn\Coker{Coker} \opn\Am{Am}
\opn\Hom{Hom} \opn\Tor{Tor} \opn\Ext{Ext} \opn\End{End}
\opn\Aut{Aut} \opn\id{id}

\opn\nat{nat}
\opn\pff{pf}
\opn\Pf{Pf} \opn\GL{GL} \opn\SL{SL} \opn\mod{mod} \opn\ord{ord}
\opn\Gin{Gin} \opn\Hilb{Hilb}\opn\sort{sort}
\opn\PF{PF}\opn\Ap{Ap}
\opn\dist{dist}
\opn\aff{aff}
\opn\relint{relint} \opn\st{st}
\opn\lk{lk} \opn\cn{cn} \opn\core{core} \opn\vol{vol}  \opn\inp{inp} \opn\nilpot{nilpot}
\opn\link{link} \opn\star{star}\opn\lex{lex}\opn\set{set}
\opn\width{wd}
\opn\Fr{F}
\opn\QF{QF}
\opn\G{G}
\opn\type{type}\opn\res{res}
\opn\conv{conv}
\opn\sr{sr}
\opn\gr{gr}

\def\pot#1#2{#1[\kern-0.28ex[#2]\kern-0.28ex]}

\opn\dirlim{\underrightarrow{\lim}}
\opn\inivlim{\underleftarrow{\lim}}
\def\Implies{\ifmmode\Longrightarrow \else
	\unskip${}\Longrightarrow{}$\ignorespaces\fi}
\def\implies{\ifmmode\Rightarrow \else
	\unskip${}\Rightarrow{}$\ignorespaces\fi}
\def\iff{\ifmmode\Longleftrightarrow \else
	\unskip${}\Longleftrightarrow{}$\ignorespaces\fi}

\let\:=\colon
\newtheorem{Theorem}{Theorem}[section]
\newtheorem{Lemma}[Theorem]{Lemma}

\newtheorem{Remark}[Theorem]{Remark}

\let\epsilon\varepsilon
\let\kappa=\varkappa
\def\qed{\ifhmode\textqed\fi
	\ifmmode\ifinner\hfill\quad\qedsymbol\else\dispqed\fi\fi}
\def\textqed{\unskip\nobreak\penalty50
	\hskip2em\hbox{}\nobreak\hfill\qedsymbol
	\parfillskip=0pt \finalhyphendemerits=0}
\def\dispqed{\rlap{\qquad\qedsymbol}}

\begin{document}

	\title{Forests whose matching powers are linear}
	\author{Nursel Erey, Antonino Ficarra}
	
	\address{Nursel Erey, Gebze Technical University, Department of Mathematics, 41400 Gebze, Kocaeli, Turkey}
	\email{nurselerey@gtu.edu.tr}
	
	\address{Antonino Ficarra, Department of mathematics and computer sciences, physics and earth sciences, University of Messina, Viale Ferdinando Stagno d'Alcontres 31, 98166 Messina, Italy}
	\email{antficarra@unime.it}
	
	
	
	
	\maketitle
	
	\begin{abstract}
      In this note, we classify all the weighted oriented forests whose edge ideals have the property that one of their matching powers has linear resolution.
	\end{abstract}
	
	\section*{Introduction}
	
	Let $G$ be a finite simple graph on vertex set $V(G)=[n]=\{1,\dots,n\}$ and edge set $E(G)$. Let $S=K[x_1,\dots,x_n]$ be the polynomial ring over a field $K$. The \textit{edge ideal} of $G$ is the monomial ideal $I(G)=(x_ix_j:\{i,j\}\in E(G))$. A fashionable topic in combinatorial commutative algebra consists of studying the interplay between the algebraic properties of $I(G)$ and the combinatorial properties of $G$.
	
	In \cite{EHHM2022a}, motivated by the classical theory of matchings, the authors introduced the concept of \textit{squarefree power}. Let $I\subset S$ be a squarefree monomial ideal. The $k$th squarefree power $I^{[k]}$ of $I$ is the monomial ideal whose generators are the squarefree generators of $I^k$. When $I=I(G)$, the generators of $I^{[k]}$ are the monomials $(x_{i_1}x_{j_1})\cdots(x_{i_k}x_{j_k})$ such that $M=\{\{i_1,j_1\},\dots,\{i_k,j_k\}\}$ is a $k$-matching of $G$. Recall that a $k$-matching of $G$ is a subset $M\subset E(G)$ of size $k$ such that $e\cap e'=\emptyset$ for all $e,e'\in  M$ with $e\ne e'$. We denote by $\nu(G)$ the \textit{matching number} of $G$ which is the maximum size of a matching of $G$.
	
	In \cite{EF}, we extended the concept of squarefree power to any monomial ideal, see also \cite{BHZN18,CFL,EH2021,EHHM2022a,EHHM2022b,FPack2,SASF2022,SASF2023,SASF2023-2}. Let $I\subset S$ be any monomial ideal, not necessarily squarefree. The $k$th \textit{matching power} $I^{[k]}$ of $I$ is the monomial ideal generated by the products $u_1\cdots u_k$, where $u_1,\dots,u_k$ is a monomial regular sequence contained in $I$. That is, $u_1,\dots,u_k\in I$ and $\supp(u_i)\cap\supp(u_j)=\emptyset$. Here, for a monomial $w\in S$, we set $\supp(w)=\{i:x_i\ \textup{divides}\ w\}$. When $I$ is squarefree, the matching powers of $I$ coincide with the squarefree powers of $I$. The biggest $k$ such that $I^{[k]}$ is called the \textit{monomial grade} of $I$ and is denoted by $\nu(I)$. Notice that $\nu(I(G))=\nu(G)$.
    
    Introducing this new concept allows us to consider edge ideals of weighted oriented graphs besides those of simple graphs. 
    
    A (\textit{vertex})-\textit{weighted oriented graph} $\mathcal{D}=(V(\mathcal{D}),E(\mathcal{D}),w)$ consists of an underlying graph $G$, with $V(\mathcal{D})=V(G)=[n]$, on which each edge is given an orientation and it is equipped with a \textit{weight function} $w:V(G)\rightarrow\mathbb{Z}_{\ge1}$, see \cite{HLMRV,PRT}. The \textit{weight} $w(i)$ of a vertex $i\in V(G)$ is denoted by $w_i$. The directed edges of $\mathcal{D}$ are denoted by pairs $(i,j)\in E(\mathcal{D})$ to reflect the orientation, hence $(i,j)$ represents an edge directed from $i$ to $j$. The \textit{edge ideal} of $\mathcal{D}$ is defined as the ideal
    $$
    I(\mathcal{D})\ =\ (x_ix_j^{w_j}\ :\ (i,j)\in E(\mathcal{D})).
    $$
    If $w_i=1$ for all $i\in V(G)$, then $I(\mathcal{D})=I(G)$ is the usual edge ideal of $G$.
    
    Note that the ideal $I(\mathcal{D})^{[k]}$ is generated by the products $u=u_1\cdots u_k$, where $u_p=x_{i_p}x_{j_p}^{w_{j_p}}\in G(I(\mathcal{D}))$, such that $M=\{\{i_1,j_1\},\dots,\{i_k,j_k\}\}$ is a matching of $G$. This observation justifies the choice to name $I^{[k]}$ the $k$th matching power of $I$. We set $\nu(\mathcal{D})=\nu(G)$ and notice that $\nu(I(\mathcal{D}))=\nu(\mathcal{D})$.
    
    In \cite{EF}, we deeply studied the matching powers of $I(\mathcal{D})$. In particular, we dealt with the problem of characterizing those matching powers $I(\mathcal{D})^{[k]}$ which are linearly related or have linear resolution.
    
    
    The note proceeds as follows. In Section \ref{sec:1-EreyFic}, we give a new proof of one of our main results from \cite{EF}. Namely we reprove that the last non-vanishing matching power of a quadratic monomial ideal is polymatroidal. Section \ref{sec:2-EreyFic} summarizes some of the results we use from \cite{EF}. They are used to prove the main result of Section \ref{sec:3-EreyFic} which is Theorem \ref{Thm:ForestPolym}. This result characterizes all weighted oriented forests $\mathcal{D}$ such that $I(\mathcal{D})^{[k]}$ has linear resolution for some $k$.
     
	\section{The last matching power of a quadratic monomial ideal}\label{sec:1-EreyFic}
	
	For a monomial $u\in S$, we set $\deg_{x_i}(u)=\max\{j:x_i^j\ \textup{divides}\ u\}$.
	
	Let $I\subset S$ be a monomial ideal. We denote by $G(I)$ the minimal monomial generating set of $I$. The ideal $I$ is called \textit{polymatroidal} if it is generated in a single degree and the  \textit{exchange property} holds: for all $u,v\in G(I)$ and all $i$ with $\deg_{x_i}(u)>\deg_{x_i}(v)$ there exists $j$ such that $\deg_{x_j}(u)<\deg_{x_j}(v)$ and $x_j(u/x_i)\in G(I)$. A squarefree polymatroidal ideal is called \textit{matroidal}.
	
	In \cite[Theorem 1.7]{EF} we proved that $I^{[\nu(I)]}$ is a polymatroidal ideal, for any quadratic monomial ideal $I\subset S$. The proof of this fact is based on a well-known result of Edmonds and Fulkerson (see \cite[Theorem 1 on page 246]{W}) which states that $I(G)^{[\nu(G)]}$ is polymatroidal for any graph $G$. In particular, $I(G)^{[\nu(G)]}$ has linear resolution, which was independently proved by  Bigdeli et al. in \cite[Theorem 4.1]{BHZN18}.
	
	In this section, we offer a new proof of the theorem of Edmonds and Fulkerson.
 \begin{Theorem}\label{Thm:theoremI(G)[nu(G)]Polymatroidal}
		Let $G$ be a finite simple graph. Then $I(G)^{[\nu(G)]}$ is polymatroidal.
	\end{Theorem}
    \begin{proof}
    	Set $k=\nu(G)$, and let $u,v\in G(I(G)^{[k]})$ and $i$ such that $\deg_{x_i}(u)>\deg_{x_i}(v)$. Our job is to find $j$ such that $\deg_{x_j}(u)<\deg_{x_j}(v)$ and $x_j(u/x_i)\in G(I(G)^{[k]})$.
    	
    	Since $\nu(G)=\nu(I(G))$, we have
    	$$
    	u={\bf x}_{e_1}\cdots{\bf x}_{e_k}\ \ \ \textup{and}\ \ \ v={\bf x}_{f_1}\cdots{\bf x}_{f_k},
    	$$
    	where $M_u=\{e_1,\dots,e_k\}$ and $M_v=\{f_1,\dots,f_k\}$ are $k$-matchings of $G$. Up to relabelling, we have $e_1=\{i,h\}$ for some $h\in[n]$. Since $\deg_{x_i}(u)>\deg_{x_i}(v)$ and $u$ and $v$ are squarefree, it follows that $i\notin V(M_v)$. Thus $h\in V(M_v)$, otherwise $\{e_1,f_1,\dots,f_k\}$ would be a $(k+1)$-matching of $G$, against the fact that $k=\nu(G)$. Thus, we may assume that $f_1=\{h,i_1\}$ for some vertex $i_1\ne h$.
    	
    	Suppose that $i_1\notin V(M_u)$. Then we have $\deg_{x_{i_1}}(u)<\deg_{x_{i_1}}(v)$ and
    	$$
    	x_{i_1}(u/x_i)=(x_{h}x_{i_1}){\bf x}_{e_2}\cdots {\bf x}_{e_k}\in G(I(G)^{[k]}).
    	$$
    	The exchange property holds in this case.
    	
    	Otherwise, if $i_1\in V(M_u)$, then we may assume that $e_2=\{i_1,j_1\}$ for some vertex $j_1\notin\{i,h\}$. Then, $j_1$ must be in $V(M_{v})$, otherwise $\{\{i,h\},\{i_1,j_1\},f_2,\dots,f_k\}$ would be a $(k+1)$-matching of $G$, which is absurd. Hence, we may assume that $f_2=\{j_1,i_2\}$ for some $i_2\notin\{i,h,i_1,j_1\}$. Now, we distinguish two more cases.
    	
    	Suppose that $i_2\notin V(M_u)$. Then we have $\deg_{x_{i_2}}(u)<\deg_{x_{i_2}}(v)$ and
    	$$
    	x_{i_2}(u/x_i)=(x_{h}x_{i_1})(x_{j_1}x_{i_2}){\bf x}_{e_3}\cdots {\bf x}_{e_k}\in G(I(G)^{[k]}).
    	$$
    	Thus, we are finished in this case.
    	
    	Otherwise, if $i_2\in V(M_u)$, then we may assume that $e_3=\{i_2,j_2\}$ for some vertex $j_2\notin\{i,h,i_1,j_1,i_2\}$. Arguing as before, we obtain that $j_2\in V(M_v)$, and we can assume that $f_3=\{j_2,i_3\}$ for some vertex $i_3\notin\{i,h,i_1,j_1,i_2\}$.\medskip
    	
    	Iterating this argument, we obtain at the $p$th step that
    	\begin{enumerate}
    		\item[(i)] $e_1=\{i,h\}$, $e_2=\{i_1,j_1\}$, $\dots$, $e_{p}=\{i_{p-1},j_{p-1}\}$ and
    		\item[(ii)] $f_1=\{h,i_1\}$, $f_2=\{j_1,i_2\}$, $\dots$, $f_p=\{j_{p-1},i_p\}$.
    	\end{enumerate}\smallskip
    
        Thus, if $i_p\notin V(M_u)$, then $\deg_{x_{i_p}}(u)<\deg_{x_{i_p}}(v)$ and
        $$
        x_{i_p}(u/x_i)={\bf x}_{f_1}\cdots {\bf x}_{f_p}{\bf x}_{e_{p+1}}\cdots{\bf x}_{e_k}\in G(I(G)^{[k]}).
        $$
        The exchange property holds in such a case.
        
        Otherwise, if $i_p\in V(M_u)$, then $e_{p+1}=\{i_p,j_p\}$ for some vertex $j_p$ different from all vertices $i,h,i_1,j_1,\dots,i_{p-1},j_{p-1},i_p$, and $f_{p+1}=\{j_p,i_{p+1}\}$ for some vertex $i_{p+1}$.\smallskip
        
        It is clear that the process described in (i)--(ii) terminates at most after $k$ steps. If we reach the $k$th step, then $\deg_{x_{i_k}}(u)<\deg_{x_{i_k}}(v)$ and
        $$
        x_{i_k}(u/x_i)={\bf x}_{f_1}\cdots {\bf x}_{f_k}=v\in G(I(G)^{[k]}).
        $$
        Thus, the exchange property holds in any case and $I(G)^{[k]}$ is polymatroidal.
    \end{proof}\medskip

\section{Linearly related matching powers}\label{sec:2-EreyFic}

In this section we summarize some of the main results from \cite{EF}.

Let $\mathcal{D}=(V(\mathcal{D}),E(\mathcal{D}),w)$ be a vertex-weighted oriented graph.
	
\begin{Remark}\label{remark: assumption on sources}
	\rm If $i\in V(G)$ is a \textit{source}, that is a vertex such that $(j,i)\notin E(\mathcal{D})$ for all $j$, then $\deg_{x_i}(I(\mathcal{D}))=1$. Therefore, without loss of generality, hereafter we assume that $w_i=1$ for all sources $i\in V(G)$.
\end{Remark}

The following result was shown in \cite[Lemma 2.4]{EF}.
\begin{Lemma}\label{lem:induced subgraph}
	Let $\mathcal{D}'$ be an induced weighted oriented subgraph of $\mathcal{D}$. Then 
	\begin{enumerate}
		\item[\textup{(a)}] $\beta_{i,{\bf a}}(I(\mathcal{D}')^{[k]}) \leq \beta_{i,{\bf a}}(I(\mathcal{D})^{[k]})$ for all $i$ and ${\bf a}\in \mathbb{Z}^n$.
		\item[\textup{(b)}] $\reg(I(\mathcal{D}')^{[k]}) \leq \reg(I(\mathcal{D})^{[k]})$.
	\end{enumerate}
\end{Lemma}

The next result summarizes \cite[Lemma 3.7, Theorem 3.8, Theorem 3.10]{EF}.

\begin{Theorem}\label{Thm:Summarize}
	Let $\mathcal{D}$ be a weighted oriented graph with underlying graph $G$. Suppose that $I(\mathcal{D})\ne I(G)$ and that $G$ has no even cycles.
	\begin{enumerate}
		\item[\textup{(i)}] If $I(\mathcal{D})^{[k]}$ is linearly related, then:
		\begin{enumerate}
			\item[\textup{(a)}] For any $u\in G(I(\mathcal{D})^{[k]})$ and any $x$ variable such that $\deg_{x}(u)=r>1$, we have $\deg_{x}(v)=r$ for every $v\in G(I(\mathcal{D})^{[k]})$.
			\item[\textup{(b)}] $k=\nu(G)$.
		\end{enumerate}\smallskip
	    \item[\textup{(ii)}] The following conditions are equivalent.
	    \begin{enumerate}
	    	\item[\textup{(a)}] $I(\mathcal{D})^{[k]}$ is linearly related.
	    	\item[\textup{(b)}] $I(\mathcal{D})^{[k]}$ is polymatroidal.
	    	\item[\textup{(c)}] $I(\mathcal{D})^{[k]}$ has a linear resolution.
	    \end{enumerate}
	\end{enumerate}
\end{Theorem}


A forest is a graph which has no cycles. In particular, the above theorem applies to all weighted oriented graphs whose underlying graphs are forests.

\section{Forests whose last matching power is polymatroidal}\label{sec:3-EreyFic}

In this section, we combinatorially classify the weighted oriented forests $\mathcal{D}$ whose last matching power $I(\mathcal{D})^{[\nu(I(\mathcal{D}))]}$ is polymatroidal.

To state the classification, we recall some concepts. A \textit{leaf} $v$ of a graph $G$ is a vertex incident to only one edge. Any tree with at least one edge possesses at least two leaves. Let $a\in V(G)$ be a leaf and $b$ be the unique neighbor of $a$. Following \cite{EH2021}, we say that $a$ is a \textit{distant leaf} if at most one of the neighbors of $b$ is not a leaf. In this case, we say that $\{a,b\}$ is a \textit{distant edge}. It is proved in \cite[Proposition 9.1.1]{J2004} (see, also, \cite[Lemma 4.2]{EH2021} or \cite[Proposition 2.2]{CFL}) that any forest with at least one edge has a distant leaf.

We say that an edge $\{a,b\}$ of a graph $G$ is an \textit{isolated edge} if $a$ and $b$ are leaves. In this case $I(G)=I(G\setminus\{a,b\})+(ab)$.

Suppose that $G$ is a forest which has at least one non-isolated edge. Then, the above result \cite[Proposition 9.1.1]{J2004} implies that we can find vertices $a_1,\dots,a_t,b,c$, with $t\ge1$, such that $a_1,\dots,a_t$ are distant leaves and $\{a_1,b\},\dots,\{a_t,b\},\{b,c\}\in E(G)$. In this case we say that $(a_1,\dots,a_t\ |\ b,c)$ is a \textit{distant configuration} of the forest $G$. Figure \ref{fig:1} displays this situation. 	

 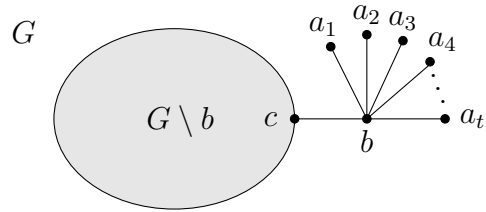
\begin{figure}[H]
		\begin{tikzpicture}[scale=0.8]
		\filldraw[fill=black!10!white] (0,0) ellipse (2cm and 1.5cm);
		\filldraw (2,0) circle (2pt) node[left,xshift=-2.5]{$c$};
		\filldraw (3.2,0) circle (2pt) node[below]{$b$};
		\filldraw (4.5,0) circle (2pt) node[right,xshift=1.5]{$a_t$};
		\filldraw (2.6,1.2) circle (2pt) node[above,xshift=-2.8]{$a_1$};
		\filldraw (3.2,1.4) circle (2pt) node[above]{$a_2$};
		\filldraw (3.8,1.3) circle (2pt) node[above]{$a_3$};
            \filldraw (4.4,-0.05) node[above]{\rotatebox[origin=c]{290}{$\dots$}};
		\filldraw (4.25,0.95) circle (2pt) node[above,xshift=4.8]{$a_4$};
		\draw[-] (3.2,0)--(2,0);
		\draw[-] (4.5,0)--(3.2,0);
		\draw[-] (3.2,0)--(2.6,1.2);
		\draw[-] (3.2,0)--(3.2,1.4);
		\draw[-] (3.2,0)--(3.8,1.3);
		\draw[-] (3.2,0)--(4.3,0.95);
		\filldraw (-2.5,1.8) node[below]{$G$};
		\filldraw (0.1,0.4) node[below]{$G\setminus b$};
		\end{tikzpicture}
		\vspace{4mm}\caption{A forest $G$ with distant configuration $(a_1,\dots,a_t\ |\ b,c)$.}
            \label{fig:1}
	\end{figure}
	
Let $\mathcal{D}$ be a weighted oriented graph with underlying graph $G$. If $W\subset V(\mathcal{D})$, we denote by $\mathcal{D}\setminus W$ the induced weighted oriented subgraph of $\mathcal{D}$ on the vertex set $V(\mathcal{D})\setminus W$. For any edge $\{a,b\}\in E(G)$, we set
$$
{\bf x}_{\{a,b\}}^{(\mathcal{D})}\ =\ \begin{cases}
    x_ax_b^{w(b)}&\textit{if}\ (a,b)\in E(\mathcal{D}),\\
    x_bx_a^{w(a)}&\textit{if}\ (b,a)\in E(\mathcal{D}).
\end{cases}
$$

We say that $\{a,b\}\in E(G)$ is a \textit{strong edge} if $\{a,b\}$ belongs to all matchings of $G$ having maximal size $\nu(G)$. In such a case, $I(\mathcal{D})^{[\nu(G)]}={\bf x}_{\{a,b\}}^{(\mathcal{D})}I(\mathcal{D}\setminus\{a,b\})^{[\nu(G)-1]}$. It is clear that an isolated edge is a strong edge.

\begin{Lemma}\label{Lemma:StrongNursel}
    Let $G$ be a forest with distant configuration $(a_1,\dots,a_t\ |\ b,c)$ and with $\nu(G)\ge2$. Then $\{a_i,b\}$ is a strong edge of $G$, for some $i$, if and only if, $t=1$ and $c\in V(M)$ for all $(\nu(G)-1)$-matchings $M$ of $G\setminus\{b\}$.
\end{Lemma}
\begin{proof}
    Suppose that $\{a_i,b\}$ is a strong edge for some $i$. Then $t=1$. Indeed, let $M$ be a matching of $G$ of size $\nu(G)$. Then $\{a_i,b\}\in M$. But, if $t>1$ then for some $j\ne i$, $(M\setminus\{\{a_i,b\}\})\cup\{\{a_j,b\}\}$ would also be a matching of $G$ of maximal size not containing $\{a_i,b\}$, which is absurd. Thus $t=1$. Now, suppose that there exists a $(\nu(G)-1)$-matching $M$ of $G\setminus b$ with $c\notin V(M)$. Then $M\cup\{\{b,c\}\}$ would be a maximum matching of $G$ not containing $\{a_i,b\}$, which is absurd.

    Conversely, assume that $(a\ |\ b,c)$ is a distant configuration of $G$ and $c\in V(M)$, for all $(\nu(G)-1)$-matchings of $G\setminus b$. Note that every matching $N$ of $G$ of size $\nu(G)$ contains either $\{b,c\}$ or $\{a,b\}$. But if $N$ contains $\{b,c\}$, then $N\setminus\{\{b,c\}\}$ would be a $(\nu(G)-1)$-matching of $G\setminus b$ whose vertex set does not contain $c$, against our assumption. The conclusion follows. 
\end{proof}

Let  $\mathcal{D}$ be a weighted oriented graph whose underlying graph $G$ is a forest. If $I(\mathcal{D})=I(G)$, then the problem of characterizing those matching powers which have linear resolution was solved in \cite[Theorem 41]{EH2021}. If $I(\mathcal{D})\ne I(G)$ and $I(\mathcal{D})^{[k]}$ has linear resolution, Theorem \ref{Thm:Summarize}(i)(b) implies that $k=\nu(G)$. Hence in the following theorem, we consider right away only the matching power $I(\mathcal{D})^{[\nu(G)]}$.

\begin{Theorem}\label{Thm:ForestPolym}
    Let $\mathcal{D}$ be a weighted oriented graph whose underlying graph $G$ is a forest, with $\nu(G)\ge2$. Suppose that $I(\mathcal{D})\ne I(G)$. Then, the following conditions are equivalent.
    \begin{enumerate}
        \item[\textup{(a)}] $I(\mathcal{D})^{[\nu(G)]}$ is linearly related.
        \item[\textup{(b)}] $I(\mathcal{D})^{[\nu(G)]}$ is polymatroidal.
        \item[\textup{(c)}] $I(\mathcal{D})^{[\nu(G)]}$ has a linear resolution.
        \item[\textup{(d)}] One of the following conditions holds:
    \end{enumerate}
    \begin{enumerate}
        \item[\textup{(d-1)}] $G$ has an isolated edge $\{a,b\}$ such that $I(\mathcal{D}\setminus\{a,b\})^{[\nu(G)-1]}$ is polymatroidal, and $$I(\mathcal{D})^{[\nu(G)]}={\bf x}_{\{a,b\}}^{(\mathcal{D})}I(\mathcal{D}\setminus\{a,b\})^{[\nu(G)-1]}.$$
        \item[\textup{(d-2)}] $G$ has a distant configuration $(a\ |\ b,c)$ with $\{a,b\}\in E(G)$ a strong edge of $G$, $I(\mathcal{D}\setminus\{a,b\})^{[\nu(G)-1]}$ is polymatroidal, and $$I(\mathcal{D})^{[\nu(G)]}={\bf x}_{\{a,b\}}^{(\mathcal{D})}I(\mathcal{D}\setminus\{a,b\})^{[\nu(G)-1]}.$$
        \item[\textup{(d-3)}] $G$ has a distant configuration $(a_1,\dots,a_t\ |\ b,c)$, $w(a_1)=\dots=w(a_t)=1$, and $I(\mathcal{D}\setminus\{b\})^{[\nu(G)-1]}$ is polymatroidal. Moreover the following statements hold.
        \begin{itemize}
            \item[\textup{(d-3-i)}]  If $I(\mathcal{D}\setminus\{b,c\})^{[\nu(G)-1]}=0$, then ${\bf x}_{\{a_i,b\}}^{(\mathcal{D})}=x_{a_i}x_b^{\delta}$ with $\delta\in\{1,w(b)\}$ for all $i$, and
        \begin{equation}\label{equation d-3-i}
        I(\mathcal{D})^{[\nu(G)]}=x_{b}^{\delta}[(x_{a_1},\dots,x_{a_t})I(\mathcal{D}\setminus\{b\})^{[\nu(G)-1]}].
        \end{equation}
            \item[\textup{(d-3-ii)}] Otherwise, $I(\mathcal{D}\setminus\{b,c\})^{[\nu(G)-1]}\ne0$ is polymatroidal, $\delta=w(b)$, $w(c)=1$ and 
            \begin{equation}\label{equation d-3-ii}
            I(\mathcal{D})^{[\nu(G)]}=x_{b}^{w(b)}[(x_{a_1},\dots,x_{a_t})I(\mathcal{D}\setminus\{b\})^{[\nu(G)-1]}+x_cI(\mathcal{D}\setminus\{b,c\})^{[\nu(G)-1]}].
            \end{equation}
        \end{itemize}   
    \end{enumerate}
\end{Theorem}
\begin{proof}
    From Theorem \ref{Thm:Summarize}(ii) and Theorem \ref{Thm:Summarize}(i)(b) it follows that (a) $\iff$ (b) $\iff$ (c). To conclude the proof, we show that (b) $\iff$ (d).\smallskip

    Firstly, we show that (b) $\implies$ (d). Suppose that $I(\mathcal{D})^{[\nu(G)]}$ is polymatroidal. If $G$ has an isolated edge, then the statement (d-1) holds. Let us assume that $G$ has no isolated edge. Then $G$ contains a distant configuration $(a_1,\dots,a_t\ |\ b,c)$.
    
    Suppose that $\{a_i,b\}$ is a strong edge for some $i$. Then, Lemma \ref{Lemma:StrongNursel} implies $t=1$. Since $I(\mathcal{D})^{[\nu(G)]}$ is polymatroidal if and only if $I(\mathcal{D}\setminus\{a,b\})^{[\nu(G)-1]}$ is polymatroidal, (d-2) follows.\smallskip

    
    Suppose that $\{a_i,b\}$ is not a strong edge for all $i$. Every matching of $G$ of size $\nu(G)$ contains either $\{b,c\}$ or $\{a_i,b\}$ for some $i=1,\dots ,t$. Therefore
    \begin{equation}\label{eq:I(D)Decomp}
       I(\mathcal{D})^{[\nu(G)]}= \ \sum_{i=1}^t{\bf x}_{\{a_i,b\}}^{(\mathcal{D})}I(\mathcal{D}\setminus\{b\})^{[\nu(G)-1]}
        +\ {\bf x}_{\{b,c\}}^{(\mathcal{D})}I(\mathcal{D}\setminus\{b,c\})^{[\nu(G)-1]}.
    \end{equation}

    We claim that
    \begin{enumerate}
        \item[\textup{(i)}] $w(a_i)=1$ for all $i=1,\dots, t$ and 
        \item[\textup{(ii)}] there exists $\delta\in\{1,w(b)\}$ such that ${\bf x}_{ \{a_i,b\}}^{(\mathcal{D})}=x_{a_i}x_b^\delta$ for all $i=1,\dots ,t$,
        \item[\textup{(iii)}] if $I(\mathcal{D}\setminus\{b,c\})^{[\nu(G)-1]}\ne0$ then ${\bf x}_{\{b,c\}}^{(\mathcal{D})}=x_cx_b^{w(b)}$ and $\delta=w(b)$.
    \end{enumerate}
    Once we have proved these facts, if $I(\mathcal{D}\setminus\{b,c\})^{[\nu(G)-1]}\ne0$, equation (\ref{eq:I(D)Decomp}) combined with (i), (ii) and (iii) implies that
    $$
        I(\mathcal{D})^{[\nu(G)]}=x_{b}^{w(b)}\big[(x_{a_1},\dots,x_{a_t})I(\mathcal{D}\setminus\{b\})^{[\nu(G)-1]}+x_cI(\mathcal{D}\setminus\{b,c\})^{[\nu(G)-1]}\big].
    $$

    Since $I(\mathcal{D})^{[\nu(G)]}$ is polymatroidal by assumption, by Lemma \ref{lem:induced subgraph} applied to the graph $\mathcal{D}\setminus\{a_1,\dots,a_t\}$, it follows that $x_cI(\mathcal{D}\setminus\{b,c\})^{[\nu(G)-1]}$ has a linear resolution. By applying \cite[Theorem 1.1]{BH2013}, we obtain that
    \begin{align*}
        (I(\mathcal{D})^{[\nu(G)]}:x_{a_1}\cdots x_{a_t})\ &=\ x_{b}^{w(b)}\big[I(\mathcal{D}\setminus\{b\})^{[\nu(G)-1]}+x_cI(\mathcal{D}\setminus\{b,c\})^{[\nu(G)-1]}\big]\\
        &=\ x_b^{w(b)}I(\mathcal{D}\setminus\{b\})^{[\nu(G)-1]}
    \end{align*}
    has a linear resolution. Now, Theorem \ref{Thm:Summarize}(ii) implies that both $I(\mathcal{D}\setminus\{b,c\})^{[\nu(G)-1]}$ and $I(\mathcal{D}\setminus\{b\})^{[\nu(G)-1]}$ are polymatroidal, and so (d-3-ii) follows.

    Otherwise, if $I(\mathcal{D}\setminus\{b,c\})^{[\nu(G)-1]}=0$, then equation (\ref{eq:I(D)Decomp}) combined with (i) and (ii) implies that
    $$
        I(\mathcal{D})^{[\nu(G)]}=x_{b}^{\delta}[(x_{a_1},\dots,x_{a_t})I(\mathcal{D}\setminus\{b,c\})^{[\nu(G)-1]}]
    $$
    By a similar argument as before, $I(\mathcal{D}\setminus\{b,c\})^{[\nu(G)-1]}$ has a linear resolution. Then, Theorem \ref{Thm:Summarize}(ii) implies that it is polymatroidal and thus (d-3-i) follows.\smallskip
    
    Next, we prove (i), (ii) and (iii).\medskip

    \textit{Proof of} (i): By Remark~\ref{remark: assumption on sources} if $a_i$ is a source, then we assume $w(a_i)=1$. Assume for a contradiction that $(b, a_i)\in E(\mathcal{D})$ but $w(a_i)>1$ for some $i$. Since $\{b,c\}$ is not a strong edge, equation (\ref{eq:I(D)Decomp}) implies that we can find a generator $u$ of $I(\mathcal{D})^{[\nu(G)]}$ with $\deg_{x_{a_i}}(u)=w(a_i)>1$. Theorem \ref{Thm:Summarize}(i)(a) implies that all generators of $I(\mathcal{D})^{[\nu(G)]}$ must have $x_{a_i}$-degree equal to $w(a_i)$. Then this implies that $\{b, a_i\}$ is a strong edge which is against our assumption. So, $w(a_i)=1$ for all $i$.\hfill$\square$\medskip

    \textit{Proof of} (ii): Lemma \ref{Thm:Summarize}(i)(a) and definition of $I(\mathcal{D})$  implies that there exists a $\delta\in\{1,w(b)\}$ such that ${\bf x}_{\{a_i,b\}}^{(\mathcal{D})}=x_{a_i}x_b^{\delta}$ for all $i=1,\dots,t$.\hfill$\square$\medskip
    
    \textit{Proof of} (iii): Suppose that $I(\mathcal{D}\setminus\{b,c\})^{[\nu(G)-1]}$ is non-zero. Let $u$ be a minimal monomial generator of $I(\mathcal{D}\setminus\{b,c\})^{[\nu(G)-1]}$.  Then ${\bf x}_{ \{a_i,b\}}^{(\mathcal{D})}u$ is a minimal monomial generator of $I(\mathcal{D})^{[\nu(G)]}$ whose $x_c$-degree is zero. Theorem \ref{Thm:Summarize}(i)(a), the assumption that $I(\mathcal{D}\setminus\{b,c\})^{[\nu(G)-1]}$ is non-zero and equation (\ref{eq:I(D)Decomp}), imply that $\deg_{x_c}({\bf x}_{\{b,c\}}^{(\mathcal{D})})=1$. Next, we claim that $\delta=w(b)$. If $b$ is a source then $w(b)=1$ and there is nothing to prove. If $w(b)=1$, there is also nothing to prove. Suppose that $b$ is not a source and $w(b)>1$, then there is a vertex $d\in\{a_1,\dots,a_t,c\}$ with $(d,b)\in E(G)$ and ${\bf x}_{\{b,d\}}^{(\mathcal{D})}=x_dx_b^{w(b)}$. Equation (\ref{eq:I(D)Decomp}) then implies the existence of a generator of $I(\mathcal{D})^{[\nu(G)]}$ whose $x_b$-degree is $w(b)>1$. Theorem \ref{Thm:Summarize}(i)(a) implies that all generators of $I(\mathcal{D})^{[\nu(G)]}$ have $x_b$-degree equal to $w(b)$. Hence $\delta=w(b)$ and ${\bf x}_{\{b,c\}}^{(\mathcal{D})}=x_cx_b^{w(b)}$.\hfill$\square$\medskip
    

    We now prove that (d) $\implies$ (b). If (d-1) or (d-2) holds then (b) follows from the following fact. If $I$ is a polymatroidal ideal and $u\in S$ is a monomial, then $uI$ is again polymatroidal. Suppose that (d-3) holds. If $I(\mathcal{D}\setminus\{b,c\})^{[\nu(G)-1]}=0$, then by equation \eqref{equation d-3-i}, the ideal $I(\mathcal{D})^{[\nu(G)]}$ is a product of polymatroidal ideals. Therefore it is polymatroidal as well by \cite[Theorem 12.6.3]{HHBook2011}.
        
        Now, suppose that (d-3-ii) holds. Then $I(\mathcal{D}\setminus\{b\})^{[\nu(G)-1]}$ and $x_cI(\mathcal{D}\setminus\{b,c\})^{[\nu(G)-1]}$ are polymatroidal ideals. Hence, $(x_{a_1},\dots,x_{a_t})I(\mathcal{D}\setminus\{b\})^{[\nu(G)-1]}$ has a linear resolution, as it is the product of monomial ideals with linear resolution in pairwise disjoint sets of variables. Therefore \cite[Corollary 2.4]{FHT2009} implies that \eqref{equation d-3-ii} is a Betti splitting. Now, since
    $$
    x_cI(\mathcal{D}\setminus\{b,c\})^{[\nu(G)-1]}\subset I(\mathcal{D}\setminus\{b,c\})^{[\nu(G)-1]}\subset I(\mathcal{D}\setminus\{b\})^{[\nu(G)-1]}
    $$
    and $x_{a_i}$ do not divide any generator of $x_cI(\mathcal{D}\setminus\{b,c\})^{[\nu(G)-1]}$ and $I(\mathcal{D}\setminus\{b\})^{[\nu(G)-1]}$, for all $1\le i\le t$, we obtain that
    \begin{center}
        $(x_{a_1},\dots,x_{a_t})I(\mathcal{D}\setminus\{b\})^{[\nu(G)-1]}\cap x_cI(\mathcal{D}\setminus\{b,c\})^{[\nu(G)-1]}=$\\[0.2cm]
        $=x_c(x_{a_1},\dots,x_{a_t})I(\mathcal{D}\setminus\{b,c\})^{[\nu(G)-1]}$
    \end{center}
    and this ideal has a linear resolution. Thus \cite[Proposition 1.8]{CFts1} implies that $I(\mathcal{D})^{[\nu(G)]}$ has a linear resolution. By Theorem \ref{Thm:Summarize}(ii) it follows that $I(\mathcal{D})^{[\nu(G)]}$ is polymatroidal and (b) follows.
\end{proof}\bigskip

Inspecting the proof of Theorem \ref{Thm:ForestPolym}, we see how to construct, recursively, all weighted oriented forests $\mathcal{D}$, with a given matching number, whose last matching power $I(\mathcal{D})^{[\nu(I(\mathcal{D}))]}$ is polymatroidal. Indeed, suppose that we have constructed all weighted oriented forests $\mathcal{D}$ with $\nu(I(\mathcal{D}))=k$ and $I(\mathcal{D})^{[k]}$ polymatroidal, then, according to the three possible cases (d-1), (d-2), (d-3), we can construct all weighted oriented forests $\mathcal{H}$ with $I(\mathcal{H})^{[\nu(I(\mathcal{H}))]}$ polymatroidal and with matching number $k+1$, one bigger than the previous fixed matching number.\bigskip

Let $\mathcal{D}$ be a weighted oriented graph, whose underlying graph $G$ is a forest, such that $I(\mathcal{D})\ne I(G)$. We illustrate the above procedure.\medskip

If $\nu(G)=1$, then $G$ is a star graph, with, say, $V(G)=[m]$ and $E(G)=\{\{i,m\}:1\le i\le m-1\}$. If $I(\mathcal{D})^{[1]}=I(\mathcal{D})$ is polymatroidal, then $w_1=\dots=w_{m-1}=1$ by Theorem \ref{Thm:Summarize}(i)(a). Since $I(\mathcal{D})\ne I(G)$, then $w_m>1$ and $E(\mathcal{D})=\{(i,m):1\le i\le n-1\}$. Thus, $I(\mathcal{D})=(x_1x_m^{w_m},x_2x_m^{w_m},\dots,x_{m-1}x_m^{w_m})=x_m^{w_m}(x_1,\dots,x_{m-1})$ is polymatroidal, for it is the product of polymatroidal ideals. In this case,
 
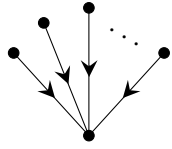
\begin{figure}[H]
 \begin{tikzpicture}
        \tikzcdset{arrow style=tikz}
	\filldraw (0,0) circle (2pt);
        \filldraw (0,1.7) circle (2pt);
        \filldraw (-0.6,1.5) circle (2pt);
        \filldraw (-1,1.1) circle (2pt);
        \filldraw (1,1.1) circle (2pt);
        \draw (0,1.7) -- (0,0) node[sloped,pos=0.5,xscale=2,yscale=2]{\arrowIn};
        \draw (-0.6,1.5) -- (0,0) node[sloped,pos=0.5,xscale=2,yscale=2]{\arrowIn};
        \draw (-1,1.1) -- (0,0) node[sloped,pos=0.5,xscale=2,yscale=2]{\arrowIn};
        \draw (1,1.1) -- (0,0) node[sloped,pos=0.5,xscale=2,yscale=2]{\arrowOut};
        \filldraw (0.5,1) node[above]{\rotatebox[origin=c]{330}{$\dots$}};
\end{tikzpicture}
\end{figure}\medskip

Now, let $\nu(G)=2$, and suppose that $I(\mathcal{D})^{[2]}$ is polymatroidal. By Theorem \ref{Thm:ForestPolym}, only one of the possibilities (d-1), (d-2), (d-3) occurs. Exploiting these three possibilities, one can see that the only weighted oriented forests $\mathcal{D}$ such that $I(\mathcal{D})^{[2]}$ is polymatroidal, are the following ones:

 
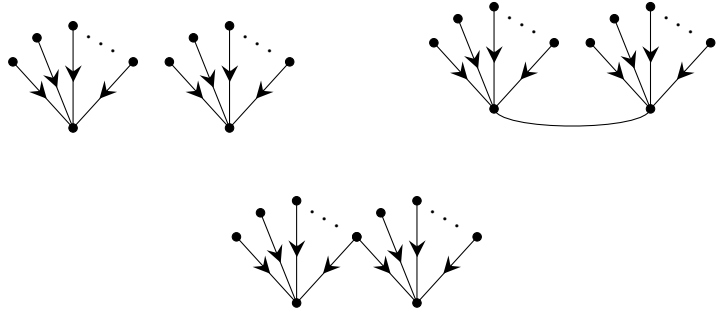
\begin{figure}[H]
 \begin{tikzpicture}[scale=0.8]
        \tikzcdset{arrow style=tikz}
	\filldraw (0,0) circle (2pt);
        \filldraw (0,1.7) circle (2pt);
        \filldraw (-0.6,1.5) circle (2pt);
        \filldraw (-1,1.1) circle (2pt);
        \filldraw (1,1.1) circle (2pt);
        \draw (0,1.7) -- (0,0) node[sloped,pos=0.5,xscale=2,yscale=2]{\arrowIn};
        \draw (-0.6,1.5) -- (0,0) node[sloped,pos=0.5,xscale=2,yscale=2]{\arrowIn};
        \draw (-1,1.1) -- (0,0) node[sloped,pos=0.5,xscale=2,yscale=2]{\arrowIn};
        \draw (1,1.1) -- (0,0) node[sloped,pos=0.5,xscale=2,yscale=2]{\arrowOut};
        \filldraw (0.5,1) node[above]{\rotatebox[origin=c]{330}{$\dots$}};
	\filldraw (2.6,0) circle (2pt);
        \filldraw (2.6,1.7) circle (2pt);
        \filldraw (2,1.5) circle (2pt);
        \filldraw (1.6,1.1) circle (2pt);
        \filldraw (3.6,1.1) circle (2pt);
        \draw (2.6,1.7) -- (2.6,0) node[sloped,pos=0.5,xscale=2,yscale=2]{\arrowIn};
        \draw (2,1.5) -- (2.6,0) node[sloped,pos=0.5,xscale=2,yscale=2]{\arrowIn};
        \draw (1.6,1.1) -- (2.6,0) node[sloped,pos=0.5,xscale=2,yscale=2]{\arrowIn};
        \draw (3.6,1.1) -- (2.6,0) node[sloped,pos=0.5,xscale=2,yscale=2]{\arrowOut};
        \filldraw (3.1,1) node[above]{\rotatebox[origin=c]{330}{$\dots$}};
\end{tikzpicture}\ \ \ \ \ \ \ \ \ \ \ \ \
\begin{tikzpicture}[scale=0.8]
        \tikzcdset{arrow style=tikz}
	\filldraw (0,0) circle (2pt);
        \filldraw (0,1.7) circle (2pt);
        \filldraw (-0.6,1.5) circle (2pt);
        \filldraw (-1,1.1) circle (2pt);
        \filldraw (1,1.1) circle (2pt);
        \draw (0,1.7) -- (0,0) node[sloped,pos=0.5,xscale=2,yscale=2]{\arrowIn};
        \draw (-0.6,1.5) -- (0,0) node[sloped,pos=0.5,xscale=2,yscale=2]{\arrowIn};
        \draw (-1,1.1) -- (0,0) node[sloped,pos=0.5,xscale=2,yscale=2]{\arrowIn};
        \draw (1,1.1) -- (0,0) node[sloped,pos=0.5,xscale=2,yscale=2]{\arrowOut};
        \filldraw (0.5,1) node[above]{\rotatebox[origin=c]{330}{$\dots$}};
	\filldraw (2.6,0) circle (2pt);
        \filldraw (2.6,1.7) circle (2pt);
        \filldraw (2,1.5) circle (2pt);
        \filldraw (1.6,1.1) circle (2pt);
        \filldraw (3.6,1.1) circle (2pt);
        \draw (2.6,1.7) -- (2.6,0) node[sloped,pos=0.5,xscale=2,yscale=2]{\arrowIn};
        \draw (2,1.5) -- (2.6,0) node[sloped,pos=0.5,xscale=2,yscale=2]{\arrowIn};
        \draw (1.6,1.1) -- (2.6,0) node[sloped,pos=0.5,xscale=2,yscale=2]{\arrowIn};
        \draw (3.6,1.1) -- (2.6,0) node[sloped,pos=0.5,xscale=2,yscale=2]{\arrowOut};
        \filldraw (3.1,1) node[above]{\rotatebox[origin=c]{330}{$\dots$}};
        \node (a) at (0,0.2) {};
        \node (b) at (2.6,0.2) {};
        \draw[-] (a)  to [out=-90,in=-90, looseness=0.4] (b);
\end{tikzpicture}
\end{figure}
\begin{figure}[H]
\begin{tikzpicture}[scale=0.8]
        \tikzcdset{arrow style=tikz}
	\filldraw (0,0) circle (2pt);
        \filldraw (0,1.7) circle (2pt);
        \filldraw (-0.6,1.5) circle (2pt);
        \filldraw (-1,1.1) circle (2pt);
        \filldraw (1,1.1) circle (2pt);
        \draw (0,1.7) -- (0,0) node[sloped,pos=0.5,xscale=2,yscale=2]{\arrowIn};
        \draw (-0.6,1.5) -- (0,0) node[sloped,pos=0.5,xscale=2,yscale=2]{\arrowIn};
        \draw (-1,1.1) -- (0,0) node[sloped,pos=0.5,xscale=2,yscale=2]{\arrowIn};
        \draw (1,1.1) -- (0,0) node[sloped,pos=0.5,xscale=2,yscale=2]{\arrowOut};
        \filldraw (0.5,1) node[above]{\rotatebox[origin=c]{330}{$\dots$}};
	\filldraw (2,0) circle (2pt);
        \filldraw (2,1.7) circle (2pt);
        \filldraw (1.4,1.5) circle (2pt);
        \filldraw (1,1.1) circle (2pt);
        \filldraw (3,1.1) circle (2pt);
        \draw (2,1.7) -- (2,0) node[sloped,pos=0.5,xscale=2,yscale=2]{\arrowIn};
        \draw (1.4,1.5) -- (2,0) node[sloped,pos=0.5,xscale=2,yscale=2]{\arrowIn};
        \draw (1,1.1) -- (2,0) node[sloped,pos=0.5,xscale=2,yscale=2]{\arrowIn};
        \draw (3,1.1) -- (2,0) node[sloped,pos=0.5,xscale=2,yscale=2]{\arrowOut};
        \filldraw (2.5,1) node[above]{\rotatebox[origin=c]{330}{$\dots$}};
\end{tikzpicture}
\vspace{4mm}
            \label{fig:2}
\end{figure}

In the second graph displayed above, the edge connecting the two bottom vertices can have an arbitrary orientation.\bigskip

Let $V(\mathcal{D})=\{a,b,c,d,e,f,g\}$, $E(\mathcal{D})=\{(c,a),(d,a),(d,b),(e,b),(f,d),(g,d)\}$, and $w:V(\mathcal{D})\rightarrow\ZZ_{\ge1}$ defined by $w(a)>1$, $w(b)>1$, and $w(v)=1$ for all other vertices $v$. Then, $\mathcal{D}=(V(\mathcal{D}),E(\mathcal{D}),w)$ is an example of weighted oriented forest with matching number $\nu(\mathcal{D})=3$ whose last matching power is polymatroidal.
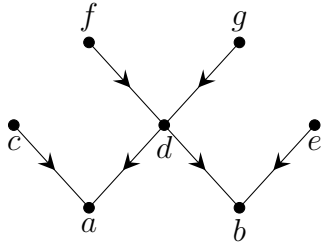
\begin{figure}[H]
	\begin{tikzpicture}
		\tikzcdset{arrow style=tikz}
		\filldraw (0,0) circle (2pt) node[below]{$a$};
		\filldraw (-1,1.1) circle (2pt) node[below]{$c$};
		\filldraw (1,1.1) circle (2pt) node[below]{$d$};
		\draw (-1,1.1) -- (0,0) node[sloped,pos=0.5,xscale=2,yscale=2]{\arrowIn};
		\draw (1,1.1) -- (0,0) node[sloped,pos=0.5,xscale=2,yscale=2]{\arrowOut};
		\filldraw (2,0) circle (2pt) node[below]{$b$};
		\filldraw (3,1.1) circle (2pt) node[below]{$e$};
		\draw (1,1.1) -- (2,0) node[sloped,pos=0.5,xscale=2,yscale=2]{\arrowIn};
		\draw (3,1.1) -- (2,0) node[sloped,pos=0.5,xscale=2,yscale=2]{\arrowOut};
		\filldraw (0,2.2) circle (2pt) node[above]{$f$};
		\filldraw (2,2.2) circle (2pt) node[above]{$g$};
		\draw (0,2.2) -- (1,1.1) node[sloped,pos=0.5,xscale=2,yscale=2]{\arrowIn};
		\draw (2,2.2) -- (1,1.1) node[sloped,pos=0.5,xscale=2,yscale=2]{\arrowOut};
	\end{tikzpicture}
	\vspace{4mm}
\end{figure}

Identifying the vertices with the variables, we obtain
\begin{align*}
	I(\mathcal{D})^{[3]}\ &=\ \big((ca^{w(a)})(fd)(eb^{w(b)}),(ca^{w(a)})(gd)(eb^{w(b)})\big)\\
	&=\ (a^{w(a)}b^{w(b)}cde)(f,g).
\end{align*}
This ideal is clearly polymatroidal because it is the product of the polymatroidal ideals $(a^{w(a)}b^{w(b)}cde)$ and $(f,g)$, see \cite[Theorem 12.6.3]{HHBook2011}.

\end{document}